\documentclass[a4paper,twoside]{article}
\usepackage[british]{babel}
\usepackage{amssymb,amsmath,amsfonts,verbatim,xkeyval,bm,upgreek}
\usepackage{graphicx}
\usepackage{subfigure}
\usepackage{enumerate}
\usepackage{lineno}
\usepackage{color}
\usepackage{hyperref}

\renewcommand{\cal}[1]{{\mathcal #1}}         
\newcommand{\RR}{{\mathbb R}}               
\newcommand{\TT}{{\mathbb T}}               
\renewcommand{\Im}{\mbox{\rm Im}\,}           

\newcommand{\diag}{\mbox{\rm diag}}

\newcommand{\iu}{{\bf i\,}}
\newtheorem{theorem}{Theorem}[section]
\newtheorem{lemma}[theorem]{Lemma}
\newtheorem{proposition}[theorem]{Proposition}
\newtheorem{corollary}[theorem]{Corollary}
\newtheorem{definition}[theorem]{Definition}
\newenvironment{proof}[1][Proof]{\begin{trivlist}
\item[\hskip \labelsep {\bfseries #1}]}{\end{trivlist}
{\nobreak\quad\nobreak\hfill\nobreak{$\blacksquare$}}}

\newenvironment{remark}[1][Remark]{\begin{trivlist}
\item[\hskip \labelsep {\bfseries #1}]}{\end{trivlist}}

\def\lie#1{\Lscr_{#1}}

\def\Lie#1{\Lscr_{#1}}
\def\reali{\mathbb{R}}
\def\complessi{\mathbb{C}}
\def\toro{\mathbb{T}}
\def\interi{\mathbb{Z}}

\def\diag{\mathop{\rm diag}}

\def\det{\mathop{\rm det}}

\def\Im{\mathop{\rm Im}}

\def\vet#1{{\bm #1}}
\def\epsilon{{\varepsilon}}

\def\phi{{\varphi}}

\def\Bscr{{\cal B}}
\def\Cscr{{\cal C}}

\def\dscr{{\bold {diag}}}

\def\Lscr{{\cal L}}

\def\Oscr{{\cal O}}
\def\Pscr{{\cal P}}

\def\Xscr{{\cal X}}

\newcommand\Xh{{X_h}}

\newcommand\DXh{\Dif\Xh}

\newcommand{\comp}{{\!\:\circ\!\,}}

\newcommand\K{{K}}
\newcommand\Knew{\K_{new}}
\newcommand\W{{W}}
\newcommand\Wnew{\W_{new}}
\newcommand\lambdanew{\lambda_{new}}

\newcommand\E{{E}}

\newcommand\EK{{\E_\K}}
\newcommand\EW{{\E_\W}}
\newcommand\GammaOb{{\Gamma_{0,\beta}}}
\newcommand\Gammaab{\Gamma_{\alpha,\beta}}

\newcommand\DK{{\Dif\K}}

\newcommand{\Dlambda}{\Dif_\lambda}
\newcommand{\Dz}{\Dif _z}
\newcommand{\Dzz}{\Dif_{z z}}
\newcommand{\Dlambdaz}{\Dif_{\lambda z}}
\newcommand\Pinv{{P_{inv}}}

\newcommand\OK{{\Omega\comp\K}}

\newcommand{\DO}{\Dif \Omega}

\newcommand{\OmegaLL}{\Omega_{LL}}
\newcommand{\OmegaLN}{\Omega_{LN}}
\newcommand{\OmegaLW}{\Omega_{LW}}
\newcommand{\OmegaNW}{\Omega_{NW}}
\newcommand{\OmegaNN}{\Omega_{NN}}
\newcommand{\OmegaWW}{\Omega_{WW}}
\newcommand{\OmegaWN}{\Omega_{WN}}

\newcommand\etaK{{\eta_\K}}
\newcommand{\etaW}{\eta_\W}

\newcommand{\hatetaWL}{\hat\eta_\W^L}
\newcommand{\hatetaWN}{\hat\eta_\W^N}
\newcommand{\hatetaWW}{\hat\eta_\W^W}

\newcommand{\xiL}{\xi^L}
\newcommand{\xiN}{\xi^N}
\newcommand{\xiW}{\xi^W}
\newcommand{\xiWL}{\xi_{\W}^L}
\newcommand{\xiWN}{\xi_{\W}^N}
\newcommand{\xiWW}{\xi_{\W}^\W}

\renewcommand\to{\rightarrow}

\newcommand\Dif{ {\mbox{\rm D}} }

\newcommand{\blue}[1]{\bgroup\color{blue}{#1}\egroup}
\newcommand{\red}[1]{\bgroup\color{red}{#1}\egroup}
\newcommand{\magenta}[1]{\bgroup\color{magenta}{#1}\egroup}
\newcommand{\green}[1]{\bgroup\color{green}{#1}\egroup}

\title{\bf A parametrization algorithm to compute lower dimensional elliptic tori in Hamiltonian systems }

\author{ {\bf Chiara Caracciolo}\\
{\small Department of Mathematics, Uppsala University}\\
{\small Department of Mathematics, University of Padua}\\
{\bf Jordi-Llu\'is Figueras}\\
{\small Department of Mathematics, Uppsala University}\\
{\bf Alex Haro}\\
{\small Departament de Matem\`atiques i Inform\`atica, Universitat de Barcelona} \\ \small{and Centre de Recerca Matem\`atica}\\
{\small e-mail: 
  {\tt chiara.caracciolo@math.unipd.it, figueras@math.uu.se, alex@maia.ub.es}}\\
}

\begin{document}

\maketitle
\tableofcontents


\selectlanguage{british}
\thispagestyle{empty}

\begin{abstract}
We present an algorithm for the construction of lower dimensional elliptic 
tori in parametric Hamiltonian systems by means of the parametrization method with 
the tangent and normal 
frequencies being prescribed. This requires that the Hamiltonian system has as many parameters as 
the dimension of the normal dynamics, and the algorithm must adjust these 
parameters. We illustrate the methodology with an implementation of the algorithm  
computing $2$--dimensional elliptic tori in a system of $4$ coupled 
pendula (4 degrees of freedom).
\end{abstract}

\bigskip


\section{Introduction}
Lower dimensional elliptic tori are invariant manifolds under the action of a Hamiltonian 
on which the internal motion conjugates to rigid quasi-periodic dynamics with a 
number of frequencies less than the number of degrees of freedom, and the complementary linear 
dynamics is linearly conjugated by harmonic oscillators. The second order normal form around 
an elliptic torus is 
\[
N=e+\sum_{i=1}^d\omega_i y_i+\dfrac{1}{2}\sum_{j=1}^{n-d} \beta_j (u_j^2+v_j^2),
\]
being $e$ a constant, $d$ the dimension of torus, $\omega= (\omega_1, \dots, \omega_d)$ its 
internal frequency vector, and $\beta= (\beta_1,\dots,\beta_{n-d})$
the normal frequency vector.

The study of the persistence of lower dimensional tori in 
nearly integrable Hamiltonian systems is one of 
the natural extensions of KAM theory~(\cite{kolmo,Arnold,moser}).
One of the first studies 
was by Melnikov~\cite{melnikov}; the first
proofs were given by Moser~\cite{moser}, Eliasson~\cite{eliasson} and Bourgain \cite{Bourgain1997}; other proofs, using a quadratic perturbative method, were given by P\"oschel, see~\cite{poschel} 
and~\cite{poschelPDE}. 

In this work we present an algorithm to compute parameterizations of lower dimensional 
elliptic tori. 
The main 
motivation of presenting this approach is given 
by the necessity to improve applicability. Indeed, in original versions of KAM theory,
results are more of pure mathematical than 
physical interest: the existence of tori is subjugated to the extreme smallness of perturbing parameters 
(see \cite{Henon66, Laskar14} where there is a good account of how small 
the distance to integrability is in classical examples like the 3 body problem). 
However, numerical explorations 
and semi-analytical procedures highlight how these invariant objects persist for 
values of the parameter way larger than the ones analytically prescribed (see \cite{Henon66}). 
In other words, there is a huge gap between the 
tori we expect to exist and the tori for which we can actually prove the 
existence using classical techniques.

In the case of full dimensional KAM tori, this gap has been dramatically reduced by applying either
the techniques based on 
the computation of the Kolmogorov normal form (see~\cite{KAM-quad} and~\cite{kam-lin} 
for a quadratic and linear convergence with respect to the small parameter) 
or on the parametrization method (firstly introduced in~\cite{Delallave}, 
see~\cite{param-book} for a review). 
The main advantage of these methods 
is that they provide algorithms which can be implemented and, therefore, used to 
produce good quality approximations of the invariant solutions.
There have been extensions 
of these techniques to the case of lower dimensional tori
(\cite{Huguet2012, gio-loc-san, Kumar2022, Fernandez-Mora2024, Haro2021, Car-2022,FontichLlaveSire,Luque-Villanueva}). 
 
One of the advantages of the parametrization method is that it 
manipulates expressions with a number of variables equal to the dimension of the tori, while 
the normal form method requires as many variables as the dimension of the phase space.
Since on one hand 
this is not a problem from an analytical point of view, it can be a computational obstacle when considering 
applications with a moderate number of degrees of freedom (see, e.g.,~\cite{Car-Loc-2021}, where 
the algorithm described was basically unusable when increasing the number of nodes to $16$ 
in the FPU model, in view of the huge memory required to represent functions, even for 
$3$-dimensional elliptic tori). Another advantage of the parametrization method is that it can 
be quite easily converted to a computer-assisted proof, once that an analytical proof 
of the convergence of the algorithm to the solution is provided. This has already been 
done for what concerns KAM Lagrangian tori in~\cite{Figueras-Haro-Luque-2017} 
and~\cite{Figueras_2020}, therefore the extension to the lower dimensional case appears quite natural. 

In this paper 
the frequency vector that describes the transverse dynamics is fixed. This is 
particularly important when the final goal is to produce a formal proof. 
Indeed, in order to perform a step of the algorithm, we require a set of 
non-resonant frequencies. In particular, the standard non-resonance conditions in 
this problem are the Diophantine and the first and second Melnikov condition. 
From a semi-analytical point of view, that is when we compute an approximated invariant solution on a computer, it is not a problem to have different frequencies at each step, because we deal with finite Fourier expansions. Therefore, we only have to check 
that we are not encountering small order resonances for the frequencies that we 
encounter. However, for the proof of the convergence we need Diophantine and 
Melnikov condition. Without fixed frequencies, we would have different 
non-resonance conditions at each step. In such a case, the convergence of 
the procedure can be proved only in a  measure sense 
(see, e.g.,~\cite{poschel,gio-loc-san,Car-2022}, and~\cite{Bia-chi-Val}, 
where the size of the resonant region is measured to be $\Oscr(\epsilon^{b_1})$ with $b_1 < 1/2$).
To fix the frequencies, we consider an initial family of Hamiltonians depending on 
a number of parameters equal to the dimension of the transverse dynamics. 
Without these free parameters, the frequencies would be slightly corrected at 
each step (or, in other words, at each step we would solve a different 
equation). One must point out, however, that the algorithm described below 
could be easily modified to, instead of moving free parameters, 
change the frequencies at each step (see for example the projection and reducibility methods
in \cite{Haro2006}). 

This article is organised as follows: in Section~\ref{sec:equations} we 
introduce the notation and set the equations that we are going to consider; 
Section~\ref{sec:step} is devoted to the description of the generic step of the 
algorithm; Section~\ref{sec:appl} contains an application of the presented 
procedure to a dynamical system made of four pendula coupled by springs, 
and in Section~\ref{sec:concl} we report the conclusions and some comments 
about future extensions and applications. In the Appendix~\ref{app} we collect some technical proposition. 

\section{Setting and invariance equations}\label{sec:equations}
\subsection{General notation}
\label{sec:general}
We denote by $\toro^n = \reali^n/(2 \pi\interi)^n$ the $n$--dimensional torus.
The spaces of $n_1\times n_2$ matrices are denoted by $\reali^{n_1 \times n_2}$. In particular we use the notation $I_n$ and $O_n$ to denote the $n\times n$ identity matrix and the zero matrix, respectively. 

Given $v \in \reali^n$, $\diag(v)$ denotes the $n\times n$ matrix with 
its diagonal the ordered elements of $v$. Also, given an $n\times n$ matrix $A$, we denote by $\diag(A)$ the matrix that contains only the diagonal elements of $A$, while all the other elements are put to zero; $\dscr(A)$ is a vector with components the diagonal of $A$.
Given a $2n\times 2n$ matrix 
$B=
\begin{pmatrix}
B_{11} & B_{12}\\
B_{21} & B_{22}
\end{pmatrix}$
($B_{ij}$ are $n\times n$ matrices), we define 
$\Pscr(B)=
\begin{pmatrix}
\diag(B_{11}) & \diag(B_{12})\\
\diag(B_{21}) & \diag(B_{22})
\end{pmatrix}.
$

Given an analytic function 
$f: \toro^{d} \rightarrow \complessi$, we write its Fourier expansion as
$$
f(\theta) = \sum_{k \in \interi^d}\hat f_k e^{\iu k\cdot \theta}, \quad 
\hat f_k = \dfrac{1}{(2\pi)^d}\int_{\toro^d} u(\theta) e^{-\iu k \cdot \theta} d\theta 
$$
and use the notation $\langle f \rangle := \hat f_0$ for the average over $\theta$.

Finally, we will assume that functions live in the (scaled) Banach space of analytic functions and 
denote the norm by $\|\cdot\|$. Moreover, given functions $f,g_1,\dots,  g_n$, 
we write that $f$ is $\Oscr(g_1,\dots, g_n)$ meaning that there exists a constant $C$
such that $\|f\|\leq C(\|g_1\|+\cdots+\|g_n\|)$.

\subsubsection{Symplectic setting}

In this paper the phase space is an open set $U \subset\reali^{2n}$ with 
coordinates $z = (z_1, \ldots, z_{2n})$. 
It is endowed with a non-degenerate exact symplectic form $\vet \omega = d \vet \alpha$ 
for a certain 1-form $\vet \alpha$.
We assume $U$ is also endowed with a Riemannian metric $\vet g$ and an anti-involutive linear isomorphism
$\vet J: {\rm T}U\rightarrow {\rm T}U$, i.e., $ \vet J ^2= -I$, such that 
$\forall \ z \in U, \ \forall u,v \in {\rm T}_z U,\ \vet \omega_z(\vet J_z u,v) = \vet g_z(u,v)$. 
It is said that $(\vet \omega,\vet g, \vet J)$ is a compatible triple and that  
$\vet J$ endows $U$ with an almost-complex structure. Hence, the anti-involution preserves both $2$-forms $\vet \omega$ and $\vet g$.  

We represent the Riemannian metric $\vet g$, the almost-complex structure $\vet J$, the 
1-form $\vet \alpha$ and the 2-form $\vet \omega$ as the matrix maps 
$G:U\to \RR^{2n\times 2n}$, $J:U\to \RR^{2n\times 2n}$, 
$a:U\to \RR^n$ and $\Omega:U\to \RR^{2n\times2n}$, given by $a(z) = (a_1(z),\ldots, a_{2n}(z))^\top$ and 
$
\Omega(z) = \Dif a(z)^\top-\Dif a(z)$ satisfying $\det \Omega(z) \neq 0$.
Since $\vet \omega$ is closed, it follows that for any triplet $(r,s,t)$ 
$$
\frac{\partial \Omega_{r,s}}{\partial z_t} +\frac{\partial \Omega_{s,t}}{\partial z_r}+\frac{\partial \Omega_{t,r}}{\partial z_s}  = 0\,.  
$$
Moreover, $\Omega^\top = -\Omega$ and it is pointwise invertible. 
The metric condition of $\vet g$ reads as $G^\top = G$ and it is positive definite. 
Furthermore,  $\Omega  J  = -G$, $ J ^\top\Omega J  = \Omega$, $ J ^\top G  J  = G$.

\begin{remark} 
The standard 1-form $\vet \alpha_z$ is $\sum_{i= 1}^n z_{n+i} {\rm d} z_{i}$ the 2-form $\vet \omega_z$ is $\sum_{i= 1}^n {\rm d} z_{n+i} \wedge {\rm d} z_{i}$; so $a(z)= (z_{n+1}, \dots, z_{2n}, 0, \ldots, 0)$ and 
$\Omega(z)=\Omega_n:=\begin{pmatrix}O_n & -I_n \\ I_n & O_n\end{pmatrix}$. 
\end{remark}

Given a function $h:U\rightarrow \reali$ the corresponding Hamiltonian vector field 
$\Xh:U \rightarrow \reali^{2n}$ satisfies $i_{\Xh}\vet \omega = -{\rm d}h$ ($\vet \omega(\cdot, \Xh)={\rm d}h$). 
In coordinates, it satisfies 
$$
\Xh(z) = (\Omega(z))^{-1} ( \Dz  h(z))^\top\,.
$$

\subsection{Invariance equations for lower dimensional invariant elliptic tori}
Consider a family of Hamiltonians 
$h_\lambda : U\subset\reali^{2n}\rightarrow \reali$, depending on a parameter $\lambda \in \reali^{n-d}$. 
We look for a lower dimensional elliptic torus of dimension $d < n$
whose tangent and normal frequency vectors are $\omega\in \reali^d$ and $\beta\in\reali^{n-d}$, respectively. Both $\omega$ and $\beta$ are fixed and $\beta_i\neq \beta_j\neq 0$ for $i\neq j$.

A parametrization of an invariant torus $ \K:\toro^d \rightarrow \reali^{2n}$  and of its 
normal bundle $\W:\toro^d \rightarrow \reali^{2n\times 2(n-d)}$ 
must satisfy the following equations:
\begin{align}
\label{eq:invariance}
\begin{cases}
\Lie{\omega}{\K}+\Xh\comp(\K;\lambda) = 0\,
\\
\Lie{\omega}{\W}+ \DXh\comp(\K; \lambda)\W - \W \GammaOb = 0\,
\end{cases},
\end{align}
where $\Lie{\omega}$ is the \emph{Lie derivative} operator in the direction of $\omega$, that given a function $u:\reali^n\rightarrow \reali$ applies as
$
\Lie{\omega}{u} = - \Dif u \cdot \omega,
$
while $\GammaOb$ is the following $2(n-d)$-dimensional matrix
\begin{equation*}
\label{eq:gammaOb}
\GammaOb = \begin{pmatrix}
O_{n-d} &  -\diag(\beta)  \\ \diag(\beta) & O_{n-d}
\end{pmatrix}.
\end{equation*}

The first equation in \eqref{eq:invariance}
is the standard \emph{invariance equation} and it tells us that the torus is 
invariant for the Hamiltonian flow associated to $h$ and the motion on 
it is the rigid rotation with frequency vector $\omega$. 
Indeed, for every $\theta_0\in\toro^d$, $\phi(t) = \K(\omega t+\theta_0)$ satisfies 
$\Xh(\K(\omega t+\theta_0))=\Xh(\phi(t))=\dot \phi(t) = 
\DK(\omega t + \theta_0)\  \omega = -\Lie{\omega}{\K}_{|\omega t + \theta_0}$.

The second equation comes from the study of the \emph{linear dynamics} normal 
to the torus, that we impose to be conjugated to a constant linear vector field with matrix $\GammaOb$.
Again, if we consider $\phi(t) = \K(\theta + \omega t)$ with $\theta\in \toro^d$ and we take $z = \phi + v$ with $v$ small, then  
$$
\dot z = \dot \phi + \dot v = \Xh(\phi+v) = \Xh(\phi) + \DXh(\phi) v + h.o.t.,
$$
where $h.o.t.$ denotes \emph{high order terms}.
Since $\dot \phi = \Xh(\phi)$ because of the invariance, at first order $v$ satisfies the following equation: 
$$
\dot v = \DXh(\phi)v\, .
$$
The particular equation for $W$ comes from the reducibility property that we require to solve the equation. The shape of the equation is explained in Section~2.2 of~\cite{Luque-Villanueva} and will be more clear in the decription of the algorithm.

In  order to design our algorithm to solve the invariance equations~\eqref{eq:invariance}, we introduce a dummy parameter $\alpha\in\RR^{n-d}$ and denote
\begin{equation*}
\label{eq:gammaab}
\Gammaab = \begin{pmatrix}
\diag(\alpha) &  -\diag(\beta)  \\ \diag(\beta) & \diag(\alpha)
\end{pmatrix}.
\end{equation*}
Then, we define a Newton-like method to construct the exact solution, starting from the approximate solution $(K, W, \lambda, \alpha)$ satisfying
\begin{align}
\label{eq:invariance-appr}
\begin{cases}
\Lie{\omega}{\K}+\Xh\comp(\K;\lambda) = E_K\,
\\
\Lie{\omega}{\W}+ \DXh\comp(\K; \lambda)\W - \W \Gammaab = E_W\,
\end{cases},
\end{align}
where the invariance errors $E_K$ and $E_W$ are ``small''. 
The role of the parameter $\lambda$ is to introduce shear in the normal dynamics (non-vanishing normal torsion), making it possible the existence of tori with the normal dynamics's frequency $\beta$. The role of $\alpha$ is to facilitate the step by making it possible to cancel some averages (more details are deferred to Subsection \ref{subsection: choice lambda alpha}). An important issue is that symplectic properties will force \emph{a posteriori} that in fact the extra-parameter $\alpha$ in $\Gammaab$ must be zero if both invariance errors equal to zero, 
see Proposition \ref{lemma:Psym}.

For the success of the procedure (and, eventually, in order to produce a proof of convergence), some hypotheses are required. 
First, we need to assume some non-degeneracy condition. One condition is usually requested in KAM theorems and relies on the fact that we are fixing the frequency vector 
$\omega$. This condition in the parametrization method is equivalent to the invertibility 
of a torsion matrix, defined later. The second condition is about how 
to choose the parameter $\lambda$. In order to do that, we must assume that the family of Hamiltonians 
is non-degenerate in some sense, to be 
detailed during the presentation of the algorithm.

The standard non-resonance hypotheses on the frequency vectors 
$\omega$ and $\beta$ that we require are the following.
\begin{definition}
\label{diof+mel}
We say that $(\omega,\beta)\in \reali^d\times \reali^{n-d}$ satisfies Diophantine-Melnikov conditions
of type $(\gamma, \tau)$, with $\gamma>0$, $\tau > d-1$, if for all $k\in \interi^d\setminus\{0\}$, $i, j= 1, \dots, n-r$ with $i\neq j$, $|\beta_i|\neq |\beta_j|\neq 0$, and:
\begin{itemize}
\item[{\rm 0)}] $|k \cdot \omega|\geq \frac{\gamma}{|k|_1^\tau}$ (Diophantine condition);
\item[{\rm 1)}] $|k \cdot \omega -\beta_i| \geq \frac{\gamma}{|k|_1^\tau}$ (first Melnikov condition);
\item[{\rm 2)}] $|k \cdot \omega -\beta_i \pm \beta_j| \geq \frac{\gamma}{|k|_1^\tau}$ (second Melnikov condition).
\end{itemize}
\end{definition}
Let us remind that these conditions are quite generic, in the sense that the set 
of non-resonant frequencies has positive Lebesgue measure.

\subsection{Reducibility}
Given the matrix maps  $V_1:\TT^d\rightarrow \reali^{2n\times m_1}, V_2:\TT^d\rightarrow \reali^{2n\times m_2}$ we define the maps $\Omega_{V_1V_2},G_{V_1V_2}:\toro^d \rightarrow\reali^{m_1\times m_2}$ as
\begin{equation*}
\label{OmegaV1V2}
\Omega_{V_1 V_2}= V_1^\top \Omega\comp K V_2
\end{equation*} 
and 
\begin{equation*}
\label{GV1V2}
G_{V_1V_2} = V_1^\top G\comp K V_2.
\end{equation*}

We define the map $P:\toro^d\to \reali^{2n\times2n}$ as the three juxtaposed 
matrix maps as follows 
\begin{equation}\label{eq:P}
P(\theta)=(L(\theta)\ N(\theta)\ \W(\theta)),
\end{equation}
where $L, \ N: \toro^d\rightarrow \reali ^{2n\times d}$ are defined as 
\begin{align}
\label{L}
L&= \DK,\\
\label{N}
N&= LA+( J \comp \K) L B + \W C,
\end{align}
with $A, B:\toro^d \rightarrow \reali^{d\times d}$ and $C:\toro^d \rightarrow \reali^{2(n-d) \times d}$ 
are defined as
\begin{equation}
\label{N-elem}
B = G_{LL}^{-1}, \quad
C =  \OmegaWW^{-1} G_{WL}\, B, \quad
A = \frac 1 2 \, C^\top\, \OmegaWW\, C.
\end{equation}

%
%
Finally, define the operator 
\begin{equation*}
\Xscr_V = \Lie{\omega}{V}+\Dz \Xh\comp (K;\lambda) V 
\end{equation*}
with $V:\TT^d\rightarrow \RR^{2n\times m}$.

In the following, we discuss properties of $P$ that are a consequence of the invariance of $K$ and $W$. 

\begin{proposition}
\label{lemma:Psym}
Under the hypothesis that both $K$ and $W$ are invariant 
(i.e. the invariance errors $E_K$ and $E_W$ appearing in \eqref{eq:invariance-appr} are 
zero) then 
\begin{enumerate}
\item $\alpha=0$;
\item $P$ is symplectic, i.e.,
\begin{equation}
\label{sym}
P^\top \OK P = \Omega_{P P}= \begin{pmatrix}
\Omega_d & O_{2d\times2( n-d)} \\ O_{2(n-d)\times 2d} & \OmegaWW 
\end{pmatrix},
\end{equation}
is invertible and
\[
P^{-1}=\Pinv := \begin{pmatrix}
N^\top \Omega\comp K\\
-L^\top \Omega \comp K\\
\OmegaWW^{-1} W^\top \Omega \comp K
\end{pmatrix};\]

\item the linearized dynamics is reducible, that is 
\begin{equation}
\label{eq:reducibility}
P^{-1}\Xscr_P = \Lambda = \begin{pmatrix}
O_{d} & T & O_{d\times 2(n-d)} \\ O_{d} & O_{d} & O_{d\times 2(n-d)} \\ O_{2(n-d)\times d} & O_{2(n-d)\times d}& \GammaOb
\end{pmatrix}\,,
\end{equation}
where $T$ is a matrix called \emph{torsion matrix} and it is defined as
\begin{equation}
\label{eq:torsion}
T := N^\top  \Omega\comp K \Xscr_N = N^\top  \Omega\comp K(\Lie{\omega}{N}+\Dz \Xh \comp (K;\lambda)N).
\end{equation}
\end{enumerate} 
\end{proposition}
The proof of this proposition  is a particular case of the more general ones 
appearing in the Appendix~\ref{app}, 
where both errors there are considered to be small but possibly non-zero.

Notice that the symplecticity of $P$ provides a simple a way to compute its inverse.
Moreover, when the torus is \emph{almost} invariant ($E_K, E_W$ are small) then:
$\alpha$ is small, 
$P$ is \emph{almost} symplectic and its inverse is approximately given by \eqref{sym},
and 
the linear dynamics are approximately conjugated to \eqref{eq:reducibility}, 
all of these up to $\Oscr(E_K, E_W)$. 
See Lemma~\ref{lemma:Psym} and Corollary~\ref{corollary:Pinverse} in the Appendix~\ref{app} for more details.

\begin{remark}
By a suitably scaling $W$ one can make the antisymmetric matrix $\OmegaWW$ that 
appears in Equation~\eqref{sym} equal to $\Omega_{2(n-d)}$.
\end{remark}

\begin{remark} The numerical computation of the term $\lie{\omega}N$ that appears in the definition of the torsion $T$ in Formula~\eqref{eq:torsion} can be rather demanding. This definition can be replaced by another one that does not involve the Lie derivative, as it has been done in~\cite{Figueras-Haro2023}. 
\end{remark}

\section{Newton step}\label{sec:step}

Given (approximated solutions) $(\K,\W,\lambda, \alpha)$ satisfying 
 Equations~\eqref{eq:invariance-appr}
with  $\EK: \toro^d \rightarrow \reali^{2n}$ and $\EW: \toro^d \rightarrow \reali^{2n\times 2(n-d)}$ 
small, we want to obtain better approximations $(\Knew,\Wnew,\lambdanew,\alpha_{new})$. For doing so, we
design
one step of the Newton method to look for corrections 
$\Delta\K:\TT^d\rightarrow \RR^{2n}$, $\Delta \W:\TT^d\rightarrow\RR^{2n\times2(n-d)}$, 
$\Delta \lambda,\Delta \alpha\in\RR^{n-d}$
such that 
\begin{align*}
\Knew & := \K+ \Delta \K,\\ \Wnew & := \W+ \Delta \W, \\ \lambdanew  & := \lambda+ \Delta \lambda ,\\ \alpha_{new} & := \alpha+ \Delta \alpha
\end{align*}
satisfy Equations \eqref{eq:invariance-appr} with errors of size $\Oscr(E_K, E_W)^2$.

\noindent
Some important facts are:
\begin{enumerate}[i)]
\item the dummy parameter $\alpha$ will be of order of the errors;
\item the corrections of the new solutions ($\Delta\K, \Delta \W, \Delta \lambda, \Delta \alpha$) 
will be of the order of the errors;
\item the invariance equation for $K$ depends only on $(\K,\lambda)$;
\item under Diophantine-Melnikov conditions on $(\omega, \beta)$, given any fixed $\lambda$ the invariance equation for $K$ can always be solved;
\item under Diophantine-Melnikov conditions on $(\omega, \beta)$ the linear dynamics equation for $W$ admits a solution for fixed values of the vectors $\lambda$ and $\alpha$.
\end{enumerate}

\begin{remark}
All these previous statements are without proofs but they will appear in a 
forthcoming paper \cite{CaraccioloFiguerasHaroTheory}. Here our focus is on the algorithm. 
\end{remark}
One starts with a set of equations for all the unknowns and the first goal is to adjust 
the parameters $\Delta\lambda$ and $\Delta\alpha$.
Once these parameters have been 
carefully chosen (see the detailed discussion in Subsection~\ref{subsection: choice lambda alpha}), we get
with a (solvable) set of cohomological equations of the type:
\begin{equation}
\label{alleq}
\begin{cases}
\Lie{\omega}{u^L} +T u^N &= v^L\\
\Lie{\omega}{u^N} &= v^N\\
\Lie{\omega}{u^W} +\Gammaab u^W &= v^W \\ 
\\
\Lie{\omega}{U^L} +T U^N - U^L \Gammaab  &= V^L \\
\Lie{\omega}{U^N} - U^N \Gammaab &= V^N\\
\Lie{\omega}{U^W}+\Gammaab U^W - U^W \Gammaab &= V^W
\end{cases}
\end{equation}
where $T$ is defined in \eqref{eq:torsion},
$v, V$ are given (there $\Delta\lambda$ and $\Delta\alpha$ have been 
absorbed) and $u, U$ are the unknowns. 
$u,v:\TT^d\rightarrow \reali^{2n}$ and $(\cdot^L, \cdot^N, \cdot^W)$ denote their components 
that are $d,d$ and $2(n-d)$ dimensional; while $U,V:\TT^d\rightarrow \reali^{2n\times 2(n-d)}$ 
decompose in $(\cdot^L, \cdot^N, \cdot^W)$ with $d,d$ and $2(n-d)$ rows, respectively.
The first two equations are the same that appear in standard KAM statements for
Lagrangian tori (see for example \cite{Figueras-Haro2023}) 
and they require a Diophantine frequency vector and the 
invertibility of the $d$-dimensional square matrix $\langle T\rangle$ 
(also known as non-degeneracy or twist conditions). The third, fourth and 
fifth equations in Equation~\eqref{alleq}, where also the matrix 
$\Gammaab$ appears, require first Melnikov condition on the frequency 
vectors $(\omega,\beta)$ in order to be solved. Finally, the 
last cohomological equation is the one that requires the second Melnikov non-resonance condition.

\begin{remark}
In Appendix \ref{app} we prove that the parameter $\alpha$ is of the size of the invariance errors and, 
so, all $\Gammaab$ in Equation \eqref{alleq} can be substituted by $\GammaOb$. From now on, we do so.
\end{remark}

In the following subsections, we describe in detail how we encounter these 
equations and which are the known terms $v,V$ appearing there.

\subsection{Correction of the approximate solution of the invariance equation for $K$}
\label{sec:K}
We look for ($\Knew$, $\lambdanew$) that 
satisfies the equation
\begin{equation*}
\Lie{\omega}{(\K+\Delta \K)}+\Xh\comp(\K+\Delta \K;\lambda+ \Delta \lambda)=0.
\end{equation*}
By Taylor expanding and discarding supralinear terms we get the following linear equation
\begin{equation}
\label{eq:stepK}
\Lie{\omega}\Delta \K + \Dz \Xh\comp(\K; \lambda)\Delta \K + \Dlambda \Xh\comp(\K; \lambda)\Delta \lambda  = - \EK.
\end{equation}
In order to solve Equation \eqref{eq:stepK} more easily 
we use the almost symplectic change of coordinates $P$ defined in \eqref{eq:P}
to write the unknown $\Delta\K$ as
\begin{equation}
\label{eq:deltak}
\Delta \K(\theta) = P(\theta) \xi_{\K}(\theta),
\end{equation}
with $\xi_{\K} : \toro^d \rightarrow \reali^{2n}$. By substituting \eqref{eq:deltak} 
in Equation~\eqref{eq:stepK} we get 
\begin{equation*}
\Lie{\omega} P \xi_{\K} + P \Lie{\omega}\xi_{\K} + \Dz \Xh\comp(\K;\lambda) P \xi_{\K} + \Dlambda \Xh\comp(\K;\lambda) \Delta \lambda  = - \EK\,,
\end{equation*}
that can be rearranged as
\begin{equation}
\label{eq:eqK with errors}
\begin{matrix}
\Lie{\omega}\xi_{\K} + P^{-1}(\Lie{\omega} P + \Dz \Xh\comp(\K;\lambda) P) \xi_{\K} + P^{-1}\Dlambda \Xh\comp(\K;\lambda)\Delta \lambda +\Oscr(E_K, E_W)^2 = \\
\Lie{\omega}\xi_{\K} + \Lambda\xi_{\K} + \Pinv\Dlambda \Xh\comp (\K;\lambda)\Delta \lambda +\Oscr(E_K, E_W)^2 = \\
- \Pinv\EK.
\end{matrix}
\end{equation}

\begin{remark}
The $\Oscr(E_K, E_W)^2$ term in Equation \eqref{eq:eqK with errors} comes 
from both the errors in the reducibility and the invertibility of $P$ times 
the norm of $\xi_{\K}$, that is $\Oscr(E_K, E_W)$.
\end{remark}

Discarding all supralinear terms and using Property \eqref{eq:reducibility with error}
the equation \emph{reduces} to 
\begin{equation}
\label{eq:inv-reduced}
\Lie{\omega}{\xi_{\K}} + \Lambda \xi_{\K}+ b \cdot \Delta \lambda = \etaK,
\end{equation}
where 
\begin{align*}
\etaK (\theta)=& \ - \Pinv(\theta)\EK(\theta)\in \reali^{2n}, \\
b(\theta) =&\ \Pinv(\theta)\Dlambda \Xh\comp(\K(\theta);\lambda)\in \reali^{2n \times (n-d)}.
\end{align*}

\subsubsection{Resolution of the invariance equation for $K$}
We split the solution $\xi_{\K}$ of Equation~\eqref{eq:inv-reduced} in two components:
\begin{equation}
\label{eq:xiK}
\xi_{\K}(\theta) = \xi_{\etaK}(\theta) - \xi_b(\theta) \Delta \lambda\,,
\end{equation}
where $\xi_{\etaK}$ and $\xi_b$ satisfy 
$\Lie{\omega}\xi_{\etaK} + \Lambda \xi_{\etaK} = \etaK $ and 
$\Lie{\omega}\xi_b + \Lambda \xi_b = b $. 
Moreover, by expressing the vectors $\xi_{\etaK}, \xi_b$ in components 
$(\cdot^L,\cdot^N, \cdot^W)$ we obtain that both triples satisfy the equations
\begin{equation}
\label{eq:invariance-split}
\begin{cases}
\Lie{\omega}{\xiL} +T \xiN  &= \eta^L_{\K}, \ b^L \\
\Lie{\omega}{\xiN} &= \eta^N_{\K},\ b^N \\
\Lie{\omega}{\xiW} +\GammaOb \xiW &= \eta^W_{\K} ,\ b^W 
\end{cases}.
\end{equation}
The first couple of cohomological equations are the same as in the Lagrangian case (see \cite{Figueras-Haro2023}). In particular, the equations with the ${\cdot}^N$ component can be solved provided that $\eta^N_{\K}$ and $b^N$ have zero average. This is the case because:
\begin{itemize}
\item 
By definition we have $\eta^N_{\K} =L^\top  \Omega \comp K \EK$.
Therefore,
\begin{align*}
\eta^N_{\K} &= L^\top  \Omega\comp K(\Lie{\omega}{\K}+ \Xh\comp(\K;\lambda))\\& = L^\top  \Omega\comp K(-\DK \cdot \omega + (\Omega\comp K)^{-1} (\Dz h\comp(\K;\lambda))^\top )
\\ &= -\OmegaLL \cdot \omega +\Dif_\theta (h\comp(\K;\lambda))^\top. 
\end{align*}
Finally, $\OmegaLL$ has zero average since the 2-form is exact, see e.g. the proof in Lemma 4.3 of~\cite{Haro-Luque-2019}), while the second term has zero average since it is a derivative of periodic functions.

\item Similarly for $b^N$, by using its definition and expressing it as a derivative of periodic functions:
\begin{align*}
b^N &= -L^\top  \Omega\comp K \ \Dlambda \Xh\comp(\K;\lambda)\\ & = -\DK^\top  \Omega\comp K \Dlambda( (\Omega\comp K)^{-1} \ (\Dz h\comp(\K;\lambda))^\top ) \\ &= -\DK^\top  \Dlambda (\Dz h\comp(\K;\lambda)  )^\top  = -\Dlambda(\Dif _\theta (h \comp(\K;\lambda)))^\top .
\end{align*}
\end{itemize}

We start by solving the $\cdot^N$ equations in \eqref{eq:invariance-split} 
(up to the average). Since it is of the form $
\Lie{\omega}{u} = v$ 
its solutions are of the form 
$$
u(\theta) = \langle u\rangle+ \sum_{k\in\interi^d\setminus\{0\}} \frac{\iu \hat v_k}{k \cdot \omega}e^{\iu  k \cdot \theta}.
$$
Then, solving the $\cdot^L$ equations is the same once the average of the $\cdot^N$ terms is fixed. This 
is done by making the cohomological equations solvable, that is, with zero average. This is accomplished 
by selecting 
\[
\begin{matrix}
\langle \xi^N_{\eta_{\K}}\rangle = \langle T \rangle ^{-1} (\langle \eta^L_{\K}- T\xi^N\rangle)\\
\langle \xi^N_{b}\rangle = \langle T \rangle ^{-1} (\langle b^L- T\xi^N\rangle)
\end{matrix}.
\]
Here, we use the non-degeneracy of $T$ (the invertibility of its average).
\begin{remark}
We notice that fixing the average of $\xi^L$ is unimportant (plays no role except a shift on the parameterization). Hence, we fix it to be equal to zero.
\end{remark}

The $\cdot^W$ equations are of the form
\begin{equation*}
\label{eq:melnikov1eq}
\Lie{\omega}{u} +\GammaOb u = v\,. 
\end{equation*}
By denoting $u = (u^1, \ldots, u^{2(n-d)}) $ this equation in terms of its Fourier coefficients is expressed as follows
\begin{equation}
\label{eq:xiK_W}
\begin{pmatrix}
- \iu  k \cdot \omega & -\beta_j \\
\beta_j & - \iu  k \cdot \omega
\end{pmatrix}
\begin{pmatrix}
\hat u_k^{j} \\ \hat u_k^{j+n-d} 
\end{pmatrix}
= \begin{pmatrix}
\hat v_k^j \\  \hat v_{k}^{j+n-d}
\end{pmatrix} \quad {\rm with}\quad  j = 1, \ldots, n-d.
\end{equation}
Notice that solutions exist as long as there are no resonances of the form $k \cdot \omega\neq \pm \beta_j$. 
Again, by imposing further the first Melnikov condition (see Definition~\ref{diof+mel}), we obtain analytic solutions.

Finally, according to Equation~\eqref{eq:xiK}, the new parametrization is
\begin{equation}
\label{eq:K-NEW}
\Knew = \K + \Delta \K = \K + P(\xi_{\etaK}-\xi_b \Delta\lambda).
\end{equation}
The term $\Delta\lambda$ will be determined in the next step, when we solve the invariance equation for $W$.

\subsection{Correction of the approximate solution for the invariance equation for $W$}
As before, we put in $(\Knew,\Wnew, \lambdanew,\alpha_{new})$ in the invariance equation for $W$ 
in \eqref{eq:invariance-appr} obtaining
\begin{equation*}
\label{eq:W}
\Lie{ \omega}{(\W + \Delta \W)}+ \Dz \Xh\comp(\K+ \Delta \K ; \lambda +\Delta \lambda) (\W+\Delta \W) - (\W+\Delta \W) \Gamma_{\alpha+\Delta \alpha,\beta} = 0.
\end{equation*}
By Taylor expanding and discarding supralinear terms (recall that also $\Gammaab$ is then replaced 
by $\GammaOb$) we get
\begin{align}
\label{eq:stepW}
\Lie{\omega}{\Delta \W}+\Dz\Xh\comp (K;\lambda) \Delta \W +\Dlambdaz \Xh\comp(K;\lambda)[\Delta \lambda, \W]\nonumber\\ + \Dzz \Xh\comp (K;\lambda)[\Delta \K, \W] -\Delta \W\GammaOb -W\Gamma_{\Delta \alpha,0} = -\EW 
\end{align}
where 
\begin{equation*}
\Dlambdaz  \Xh\comp (K;\lambda)[\cdot, \W] : \Delta \lambda \rightarrow \Delta \lambda^\top  \Dlambdaz  \Xh\comp (K;\lambda)\W ,
\end{equation*}
being $\Dlambdaz \Xh \comp (K;\lambda)$ a collection of $n-d$ squared matrices of dimension $2n$ with\footnote{$(\Xh)_j$ denotes the $j$--th component of the vector field.} 
$$
(\Dif_{\lambda_k z } \Xh\comp (K;\lambda))_{i,j} = \partial_{\lambda_k} \partial_{z_i} (\Xh\comp (K;\lambda))_j , \ \forall \ i,j = 1, \ldots, 2n, \ k=1,\ldots,n-d,
$$
and
\begin{equation*}
\Dzz \Xh\comp (K;\lambda)[\cdot, \W] : \Delta \K \rightarrow \Delta \K^\top  \Dzz \Xh \comp (K;\lambda)\W 
\end{equation*}
being $\Dif_{z z } \Xh\comp (K;\lambda) $ a collection of $2n$ squared matrices of dimension $2n$ with 
$$
(\Dif_{z_k z } \Xh\comp (K;\lambda))_{i,j} = \partial_{z_k} \partial_{z_i} (\Xh\comp (K;\lambda))_j \ \ \forall\ i,j,k = 1, \ldots, 2n.
$$

By applying the almost symplectic change of coordinates $P$,
\begin{equation*}
\Delta \W (\theta)= P(\theta) \xi_{\W}(\theta)
\end{equation*} 
with $\xi_{W}:\TT^d\rightarrow \reali^{2n \times 2(n-d)}$, and removing supralinear terms in \eqref{eq:stepW}
we obtain
\begin{equation}
\label{eq:W-almost-reduced}
\Lie{\omega}{\xi_{\W}}+\Lambda \xi_{\W} -\xi_{\W}\GammaOb - P^{-1} W 
\Gamma_{\Delta \alpha,0}+ \Bscr(\Delta \lambda)+ \Cscr (\Delta \K) = \etaW,
\end{equation}
where $\etaW,\Bscr(\Delta \lambda),\Cscr(\Delta \K):\TT^d\rightarrow \reali^{2n\times 2(n-d)}$ are defined as
\begin{align*}
\etaW(\theta) & = -\Pinv(\theta)\EW(\theta), \\ 
\Bscr(\Delta \lambda)(\theta) &= \Pinv(\theta)\Dlambdaz \Xh(K(\theta);\lambda)[\Delta \lambda, \W(\theta)], \\ 
 \Cscr(\Delta \K)(\theta) &= \Pinv(\theta)\Dzz \Xh (K(\theta);\lambda)[\Delta \K(\theta), \W(\theta)].
\end{align*}

Since $\Delta \K = P \xi_{\etaK} - P \xi_b \Delta \lambda$ with the unknown $\Delta \lambda$, 
we split the contribution of $\Cscr(\Delta \K)$ as follows:

\begin{enumerate}
\item  the known terms, i.e. $\Cscr(P \xi_{\etaK})$, are moved to the r.h.s. of Equation \eqref{eq:W-almost-reduced}, thus replacing $\etaW$ by 
\begin{equation*}
\label{eq:etaWhat}
\hat{\eta}_W = \etaW- \Cscr(P\xi_{\etaK});
\end{equation*}

\item the terms depending on $\Delta \lambda$, i.e. $-\Cscr(P \xi_b \Delta \lambda)$, are collected in 
\begin{equation*}
\label{eq:Bhat}
\hat \Bscr(\Delta \lambda) = \Bscr(\Delta \lambda)-\Cscr(P\xi_b \Delta \lambda).
\end{equation*}
\end{enumerate}
Therefore, we obtain
\begin{equation*}
\label{eq:W-reduced}
\Lie{\omega}{\xi_{\W}}+\Lambda \xi_{\W} -\xi_{\W}\GammaOb - P^{-1} W \Gamma_{\Delta \alpha,0}+ \hat \Bscr(\Delta \lambda) = \hat{\eta}_\W,
\end{equation*}
which splits into the set of equations
\begin{equation}
\label{eq:W-split}
\begin{cases}
\Lie{\omega}{\xiWL} +T \xiWN - \xiWL \GammaOb &= \hatetaWL - \hat \Bscr^L (\Delta \lambda)\\
\Lie{\omega}{\xiWN} - \xiWN \GammaOb &= \hatetaWN - \hat \Bscr^N (\Delta \lambda)\\
\Lie{\omega}{\xiWW}+\GammaOb \xiWW - \xiWW \GammaOb &= \hatetaWW - \hat \Bscr^W (\Delta \lambda)+\Gamma_{\Delta \alpha,0} 
\end{cases}\,.
\end{equation}

We start by solving the third equation in \eqref{eq:W-split}, that is of the form 
\begin{equation}
\label{eq:op2}
\Lie{\omega}{u}+\GammaOb u - u \GammaOb = v,
\end{equation}
with $u:\TT^d\rightarrow \reali^{2(n-d)\times 2(n-d)}$.
In terms of Fourier coefficient and components ($\hat u_k^{i,j}$ denotes the $k$-th Fourier coefficient of the element $i,j$ of $u$) this equation is equivalent to
\begin{equation}
\label{eq:xiW_W}
\begin{pmatrix}
- \iu  k\cdot \omega   & -\beta_j & -\beta_i & 0 \\
\beta_j & - \iu  k\cdot \omega  & 0 & -\beta_i  \\
\beta_i & 0 &  - \iu  k\cdot \omega   & -\beta_j  \\
0 & \beta_i & \beta_j & - \iu  k\cdot \omega  
\end{pmatrix}
\cdot
\begin{pmatrix}
\hat u_k^{i,j} \\
\hat u_k^{i,j+n-d} \\
\hat u_k^{i+n-d,j} \\
\hat u_k^{i+n-d,j+n-d} \\
\end{pmatrix} 
=
\begin{pmatrix}
\hat v_k^{i,j}\\
\hat  v_k^{i,j+n-d} \\
\hat  v_k^{i+n-d,j} \\
\hat  v_k^{i+n-d,j+n-d}\\
\end{pmatrix} \,,
\end{equation}
with $i,j = 0, \ldots, n-d$.
When $i \neq j$, and if $\beta_i \neq \pm\beta_j$, the equation admits a solution provided that the determinant of the matrix is not zero, i.e., assuming again a non-resonance condition $k\cdot \omega \neq \pm(\beta_i\pm \beta_j)$, in this case the second Melnikov.
Therefore, we have to discuss the existence of a solution to the third equation in \eqref{eq:W-split} 
only in the case $i=j$ and $k = 0$. This means that we have to consider the averaged terms in the diagonal of the four $n-d$ dimensional blocks. 
In matrix form, we have the following equation
\begin{equation}
\label{eq:xiW_Wk0}
\begin{pmatrix}
0   & -\beta_i & -\beta_i & 0 \\
\beta_i & 0  & 0 & -\beta_i  \\
\beta_i & 0 & 0 & -\beta_i  \\
0 & \beta_i & \beta_i & 0 
\end{pmatrix}
\cdot
\begin{pmatrix}
\hat u_0^{i,i} \\
\hat u_0^{i,i+n-d} \\
\hat u_0^{i+n-d,i} \\
\hat u_0^{i+n-d,i+n-d} \\
\end{pmatrix} 
=
\begin{pmatrix}
\hat v_0^{i,i}\\
\hat  v_0^{i,i+n-d} \\
\hat  v_0^{i+n-d,i} \\
\hat  v_0^{i+n-d,i+n-d}\\
\end{pmatrix} \,,
\end{equation}
that has a solution if and only if $\hat v_0^{i,i} = -\hat v_0^{i+n-d,i+n-d}$ and $\hat v_0^{i,i+n-d} = \hat v_0^{i+n-d,i}$, i.e., $\langle v_0\rangle$ is of the form 
\begin{equation}
\label{image}
\begin{pmatrix}
\hat v_0^{i,i}  & \hat v_0^{i,i+n-d} \\
\hat v_0^{i,i+n-d} &- \hat v_0^{i,i}
\end{pmatrix}.
\end{equation}
In our equation $\hat v_0 = \langle \hatetaWW - \hat \Bscr^W (\Delta \lambda)+\Gamma_{\Delta \alpha,0}\rangle$. In the next subsection, we discuss how to choose $\Delta \lambda$ and $\Delta \alpha$ such that $\hatetaWW - \hat \Bscr^W (\Delta \lambda)+\Gamma_{\Delta \alpha,0}  \ \in \Im(\Lscr^2_{\omega,\GammaOb})$, that is, how to make the equation solvable. Clearly, in view of the degenerate nature of the matrix, the solution is not unique. Indeed, every matrix of type 
\begin{equation*}
\label{kernel}
\begin{pmatrix}
\hat u_0^{i,i} & \hat u_0^{i,i+n-d}\\
-\hat u_0^{i,i+n-d} & \hat u_0^{i,i}
\end{pmatrix}
\end{equation*}
has zero image and therefore the solution will depend on $2$ parameters ($\hat u_0^{i,i}$ and $\hat u_0^{i,i+n-d}$) for each $i=1,\ldots,n-d$. 

Once we solve for $\Delta\lambda$ and $\Delta\alpha$ we can also solve the first and second equations. 

\subsubsection{Choice of $\Delta \lambda$ and $\Delta \alpha$}
\label{subsection: choice lambda alpha}
The values of $\Delta\lambda$ and $\Delta\alpha$ are determined while solving the third equation in Formula~\eqref{eq:W-split}. As mentioned before, we must concentrate on the $i=j$ and $k=0$ terms of the family of linear systems in \eqref{eq:xiW_W}, i.e., Equation~\eqref{eq:xiW_Wk0}. In particular, this is equivalent to study the equation 
\begin{equation}
\label{eq:lambdaalpha}
\Pscr (\GammaOb \langle\xiWW \rangle- \langle\xiWW \rangle \GammaOb + \langle \hat \Bscr^W\rangle (\Delta \lambda)-\Gamma_{\Delta \alpha,0}) = \Pscr(\langle\hatetaWW \rangle)
\end{equation}
where $\Pscr$ is defined in Section \ref{sec:general}.
Notice that $\Pscr(\label{eq:stcond}\GammaOb \langle\xiWW \rangle- \langle\xiWW \rangle \GammaOb)$ is
$$
\begin{pmatrix}
\diag(s) & \diag(t)\\
\diag(t) &-\diag(s)
\end{pmatrix}
$$
being 
\begin{equation*}
\label{diags}
\diag(s) = -\diag(\beta) \diag(\langle \xiWW \rangle_{21})-\diag(\langle \xiWW \rangle_{12})\diag(\beta)
\end{equation*}
and 
\begin{equation*}
\label{diagt}
\diag(t)=-\diag(\beta) \diag(\langle \xiWW \rangle_{22})+\diag(\langle \xiWW \rangle_{11})\diag(\beta),
\end{equation*}
where
\begin{equation*}
\langle\xiWW\rangle = 
\begin{pmatrix}
\langle\xiWW\rangle_{11} & \langle\xiWW\rangle_{12} \\ 
\langle\xiWW\rangle_{21} & \langle\xiWW\rangle_{22} \\ 
\end{pmatrix}.
\end{equation*}

Equation
\eqref{eq:lambdaalpha} rewrites as the system of equations for the unknowns $(s,t, \Delta \alpha, \Delta \lambda)$:
\begin{equation*}
\label{eq:diag terms}
\begin{pmatrix}
I_{n-d} & O_{n-d} & -I_{n-d} & \langle\hat\Bscr^W\rangle_{11}  \\
O_{n-d} & I_{n-d} & O_{n-d} &\langle \hat\Bscr^W \rangle_{12} \\
O_{n-d} & I_{n-d} & O_{n-d} & \langle\hat\Bscr^W\rangle_{21} \\
-I_{n-d} & O_{n-d} & -I_{n-d} & \langle\hat\Bscr^W\rangle_{22} 
\end{pmatrix} \begin{pmatrix}
s\\t\\ \Delta \alpha \\ \Delta \lambda
\end{pmatrix} = \begin{pmatrix}
\dscr(\langle\hat{\eta}_{W}^W\rangle_{11}) \\
\dscr(\langle\hat{\eta}_{W}^W\rangle_{12}) \\
\dscr(\langle\hat{\eta}_{W}^W\rangle_{21}) \\
\dscr(\langle\hat{\eta}_{W}^W\rangle_{22}) 
\end{pmatrix}.
\end{equation*}
Notice that it admits a solution under the \emph{trasversality condition:} 
$$\langle\hat\Bscr^W\rangle_{12}-\langle \hat\Bscr^W\rangle_{21}$$ 
is invertible. Looking at the definition of the matrix $\hat{\Bscr}$, this 
is a condition on the family of Hamiltonians. 
Under this assumption its solution is
\begin{align*}
\label{eq:solutions alpha lambda}
\begin{split}
s &= \frac{\dscr(\langle\hat{\eta}_{W11}^W\rangle- \langle\hat{\eta}_{W22}^W\rangle)}{2} -\frac{(\langle\hat\Bscr^W_{11}\rangle -\langle\hat\Bscr^W_{22}\rangle)}{2}\Delta \lambda \,,\\
t &= \frac{\dscr(\langle\hat{\eta}_{W12}^W\rangle+ \langle\hat{\eta}_{W21}^W\rangle)}{2} -\frac{(\langle\hat\Bscr^W_{12}\rangle +\langle\hat\Bscr^W_{21}\rangle)}{2}\Delta \lambda \,,\\
\Delta \alpha &= -\frac{\dscr(\langle\hat{\eta}_{W11}^W\rangle+ \langle\hat{\eta}_{W22}^W\rangle)}{2} +\frac{(\langle\hat\Bscr^W_{11}\rangle +\langle\hat\Bscr^W_{22}\rangle)}{2}\Delta \lambda\,, \\
\Delta \lambda &= (\langle\hat\Bscr^W_{12}\rangle -\langle\hat\Bscr^W_{21}\rangle)^{-1}\dscr(\langle\hat{\eta}_{W12}^W\rangle- \langle\hat{\eta}_{W21}^W\rangle)\,.\\
\end{split}
\end{align*}
Furthermore, we emphasize again that once we have a solution $(s,t,\Delta\alpha, \Delta\lambda)$ there 
is a $2(n-d)$-dimensional family of solutions for $\langle \xiWW\rangle$ that satisfies~\eqref{eq:stcond}: given a solution $\begin{pmatrix}
\langle \xiWW\rangle_{11}& \langle \xiWW\rangle_{12}\\
\langle \xiWW\rangle_{21} & \langle \xiWW\rangle_{22}\\
\end{pmatrix}$ of the Equations \eqref{diags} and~\eqref{diagt}, we obtain another solution by adding 
$\begin{pmatrix}
\diag(v_1) & \diag(v_2)\\-\diag(v_2) &\diag(v_1)
\end{pmatrix}$, with $v_1, v_2\in\mathbb R^{n-d}$.
The existence of these extra parameters $v_1, v_2$ reflects the fact that the operator defined in 
\eqref{eq:op2} has a kernel and the solution for $W$ is symmetric for rotations. 
We, then, fix these two vectors to be equal to zero.

\begin{remark}
By adding the extra condition of fixing that the diagonal and antidiagonal terms of $\Omega_{WW}$ are equal to the ones in $\Omega_{n-d}$, we could fix the vector $v_1$. 
\end{remark}

\begin{remark} The dummy correction $\Delta \alpha$ has the role to control the diagonal terms of the r.h.s. of the third equation in Equation~\eqref{eq:W-split} and make them in the form of the image~\eqref{image}. However, if the term $\langle \hat \eta_W^W - \hat\Bscr^W(\Delta\lambda)\rangle$ is antisymmetric (e.g., when the symplectic matrix $\Omega$ is constant) the $\Delta\alpha$ correction is not needed.

The $\Delta \lambda$ correction acts on the antidiagonal terms. Without considering this extra-correction, the linear transverse motion would still be elliptic but with frequency vector slightly different from $\beta$. 
\end{remark}



\subsubsection{Resolution of the first and second equation}
\label{subsubsection: 1st 2nd eq}
The first and second equations in \eqref{eq:W-split} are very similar to each other and are of the form 
$$
\Lie{\omega}{u} -u\GammaOb= v.
$$
In Fourier components it becomes
$$
\begin{pmatrix}
- \iu k \cdot \omega & -\beta_j \\
\beta_j & - \iu  k \cdot \omega
\end{pmatrix}
\begin{pmatrix}
\hat u_k^{i,j} \\ \hat u_k^{i,j+n-d}
\end{pmatrix}
= \begin{pmatrix}
\hat v_k^{i,j} \\  \hat v_k^{i,j+n-d}
\end{pmatrix},
$$
with $i=1,\ldots ,d$ and $j = 1, \ldots,n-d$.
Notice it is the same as Equation~\eqref{eq:xiK_W}, so solvable provided first Melnikov condition is satisfied. 

With this final resolution, we have determined all the corrections needed. 
We can use $\Delta \lambda$ to completely determine $\Knew$ as in~\eqref{eq:K-NEW} and $\Wnew = W + P \xi_{W}$.

\section{Application: coupled pendula}
\label{sec:appl}
We tested the  algorithm described in the previous section for computing lower dimensional tori in a system of $4$ degrees of freedom with parameters $\ell_j$ ($j=1, \ldots , 4$) and $k_j$($j= 1,\ldots 3$). Its Hamiltonian is 
\begin{equation}
\label{eq:ham-pendula}
h(y,x) = \sum_{j=1}^4\left(\frac{y_j^2}{2 \ell_j^2}-\ell_j \cos{x_j}\right) +\frac{k_1}{2}(\ell_2 x_2-\ell_1 x_1)^2 + \epsilon \sum_{j=1}^2 \frac{k_{j+1}}{2}(\ell_{j+2} x_{j+2}-\ell_{j+1} x_{j+1})^2.
\end{equation}
When $\varepsilon = 0$ the system consists of two uncoupled pendula plus
two pendula coupled together with a quadratic potential. 
In this case we choose our torus $K$ ($d=2$) to be the coupled pendula and 
the normal dynamics $W$ as the harmonic oscillators resulting in linearizing the uncoupled pendula around 
$(x_j, y_j)=0$ with frequencies $\beta_j = \dfrac1{\sqrt{\ell_{j+d}}}$, $j=1,2$.

In the Hamiltonian~\eqref{eq:ham-pendula} the dependence on a parameter $\lambda$ is not yet explicit. Our choice is the following: we take variable $\ell$ for the last two pendula, that is, we take 
$\ell_{j+2} = 1/(\beta_j+\lambda_j)^2$ for $j=1,2$. In particular, the parameter vector $\lambda$ at 
$\varepsilon=0$ is initially set to $0$.

\subsection{Initialization}
We are looking for $2$--dimensional tori whose transverse elliptic dynamics is given by small oscillations around the equilibrium positions of the other two pendula. 
Let us define the set of coordinates: we divide in different blocks the variables associated to the first $2$ pendula $(y_1, y_2, x_1, x_2)$ and the coordinates for the last $2$, $(y_3, y_4, x_3, x_4)$. 
Therefore, the symplectic matrix $\Omega$ is constant and it takes the form 
$ \Omega = \begin{pmatrix}
\Omega_2 & 0 \\ 0 & \Omega_2
\end{pmatrix}\,.
$

An initial parametrization of the torus $K$ with fixed frequency $\omega$, denoted with $K^0$, 
is given by 
$( K^0_1,\ K^0_2,\ K^0_3,\ K^0_4,\ 0,\ 0,\ 0,\ 0),
$
where $( K^0_1,\ K^0_2,\ K^0_3,\ K^0_4)$ is the parametrization of the torus coming from the
two coupled pendula. This torus is obtained as follows: First we set $k_1=0$ and get the torus
from the uncoupled system. Then, we perform a continuation with respect to $k_1$ by applying 
the techniques in \cite{Haro-Luque-2019}.

In our computations, we use Fourier expansions to represent our functions. 
We consider a sample of points on a regular grid of size $N_F = (N_{F_1}, N_{F_2})\in \mathbb{N}^2$, $\theta = (\theta_1, \theta_2) = (\frac{j_1}{N_{F_1}}, \frac{j_2}{N_{F_2}})$ and $j_1=0, \ldots, N_{F_1}$ and $j_2=0, \ldots, N_{F_2}$. In particular, (by trial and error) we set $N_{Fj} = 1024$. 

As initial approximation for $\W^0$, we first set
$$
\tilde{\W}^0 = \begin{pmatrix}
0 & 0 & 0 & 0 \\ 
0 & 0 & 0 & 0 \\ 
0 & 0 & 0 & 0 \\ 
0 & 0 & 0 & 0 \\ 
1 & 0 & -\sqrt{\ell_3^3} & 0 \\ 
0 & 1 & 0 & -\sqrt{\ell_4^3}  \\ 
1 & 0 & \frac{1}{\sqrt{\ell_3^3}} & 0 \\ 
0 & 1 & 0 & \frac{1}{\sqrt{\ell_4^3}}\\ 
\end{pmatrix},
$$
so as to satisfy the equation $\DXh(\K^0;\lambda^0)\tilde{\W}^0-\tilde{\W}^0\GammaOb = 0.$ However, 
in the procedure we are also interested in keeping fixed the symplectic 
form, that is we want to fix $\W^\top  \Omega\comp K \W = \OmegaWW  = \Omega_2$. 
This condition is not satisfied by $\tilde \W^0$ since 
$$
\tilde \W^{0T}\Omega\comp K\tilde \W^0  = \begin{pmatrix}
O_{n-d} & -\diag(a) \\ \diag(a) & O_{n-d}
\end{pmatrix},$$ 
where $a_j=\sqrt{\ell_{j+d}^3}+\dfrac{1}{\sqrt{\ell_{j+d}^3}}$. 
To fix the symplectic form $\OmegaWW $ we normalize by
$$
\W^0 = \tilde \W^0 \begin{pmatrix}
\diag(b) & O_{n-d} \\ O_{n-d} & \diag(b)
\end{pmatrix}
$$
where $b_j = 1/\sqrt{a_j}$ for $j=1,2$. 

\subsection{Results}
Here we report some numerics done in multiprecision using the library \textbf{mpfr}, see \cite{mpfrlib}, with 
60 digits of precision. 
We fix the following parameters:
$$
k_1 = 10^{-2},\ k_2 = k_3 = 1,\ \omega_1 = \sqrt{2}, \ \omega_2 = \sqrt{3},\ \beta_1 = \sqrt{2.5},\ \beta_2 = \sqrt{2.8}.
$$
Moreover, we fix $\ell_1$ and $\ell_2$ as $$
\ell_1 = 0.45678,\ \ell_2 =  0.325,
$$ 
while we define the last two as $\ell_3 = 1/\beta_1^2, \ell_4 = 1/\beta_2^2$.
The frequencies (and the $\ell$s) have been chosen so as to be non-resonant and so as to correspond to librations around the equilibrium positions (that is, we do not want to consider complete rotations of the pendula). Moreover, the $\ell$ of the last two pendula is such that the frequencies of the initial small oscillations are exactly $\beta_1$ and $\beta_2$. 

The efficiency of the procedure is tested by looking at the norms of the invariance errors $E_K$ and $E_W$, that are defined as follows:
$$
|E_K| = \max_{1 \le i \le 2n} |(E_{K})_i|\,
$$
and 
$$
|E_W| = \max_{1 \le i \le 2n} \sum_{1 \le j \le 2(n-d)}|(E_{W})_{i,j}|\,.
$$
In Table~\ref{table} we observe the positive correlation between the error of the invariance equations
and the magnitude of $\epsilon$: as it is expected, the larger $\varepsilon$ is the worse 
are the errors.
\begin{table}
\begin{center}
\begin{tabular}{c|c|c}
$\epsilon$ & $|\EK|$ & $|\EW|$  \\
\hline
$1\cdot 10^{-6}$  & $1.20531845 \cdot 10^{-41}$
& $1.01027895 \cdot 10^{-43}$\\
$1\cdot 10^{-5}$ & $1.09024078\cdot 10^{-41}$ & $1.00911089 \cdot 10^{-42}$\\
$1\cdot 10^{-4}$  & $1.01587889\cdot 10^{-41}$ & $9.93762843\cdot 10^{-42}$\\
$1\cdot 10^{-3}$ & $  8.68612723\cdot 10^{-40}$ & $1.16683524\cdot 10^{-36}$\\
\end{tabular}
\end{center}
\caption{Norms of the errors $\EK$ and $\EW$ for different values of the parameter $\epsilon$ and after $8$ iterations of the algorithm and $N_{F_1} = N_{F_2} = 1024$. }
\label{table}
\end{table}
In Figure~\ref{andamento}, we report in semi-logscale the rate of convergence of the errors $|\EW|$, $|\EK|$ and of the corrections $|\Delta \lambda|$ with respect to the number of steps. Also in this picture, one can observe how the rate of convergence becomes slower once we approach the numerical limit.
\begin{figure}[!]
\centering
\includegraphics[scale=0.9]{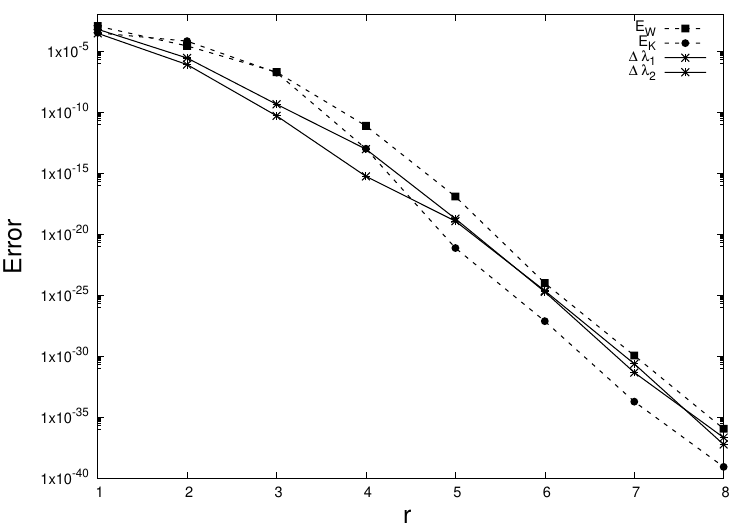}
\caption{Decrease, in semi-logscale, of the norms of the errors $\EK$ and $\EW$ and $|\Delta \lambda|$ when $\epsilon = 10^{-3}$, $r$ denotes the number of steps.}
\label{andamento}
\end{figure}

In Figures~\ref{toro} and~\ref{lindin} we show the obtained torus after $8$ iterations 
of the algorithm with parameters $k_1 = 10^{-2}$ and $\epsilon = 10^{-3}$, where $\theta$ is sampled in a grid of $N_{F_1}\times N_{F_2}$ points. 
The computational time to compute $8$ steps for $N_{F_1}=N_{F_2} = 1024$ on a computer equipped with a processor {\tt Intel(R) Core(TM) i7-8550U CPU @ 1.80GHz} and $16 GB$ of RAM is around $10$h. 
The amount of time is strongly correlated to the high number of Fourier components needed to accurately represent the torus.
\begin{figure}[!]
\centering
\includegraphics[scale=0.35]{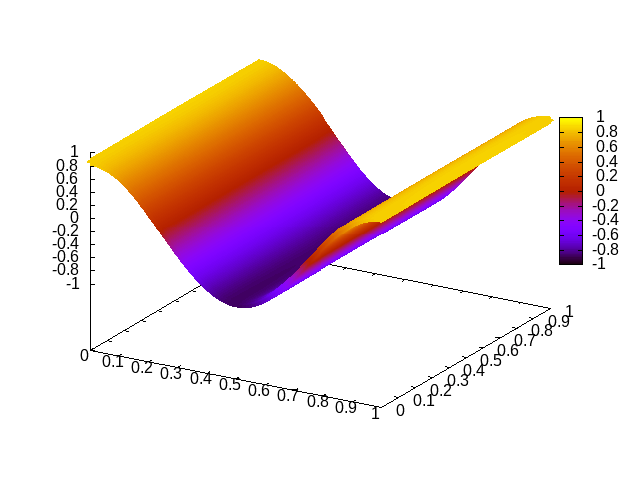}
\includegraphics[scale=0.35]{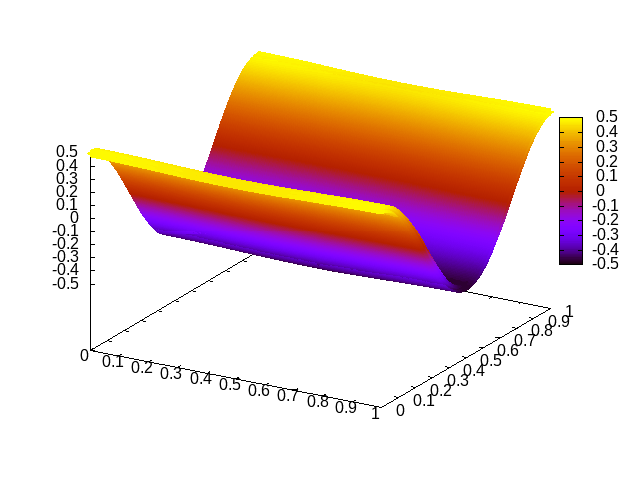}
\includegraphics[scale=0.35]{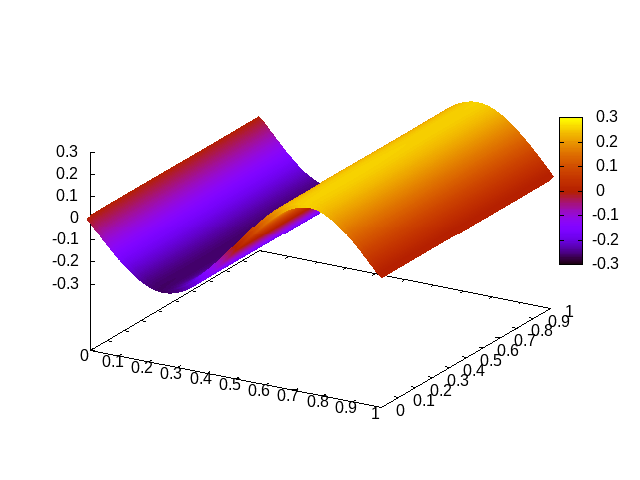}
\includegraphics[scale=0.35]{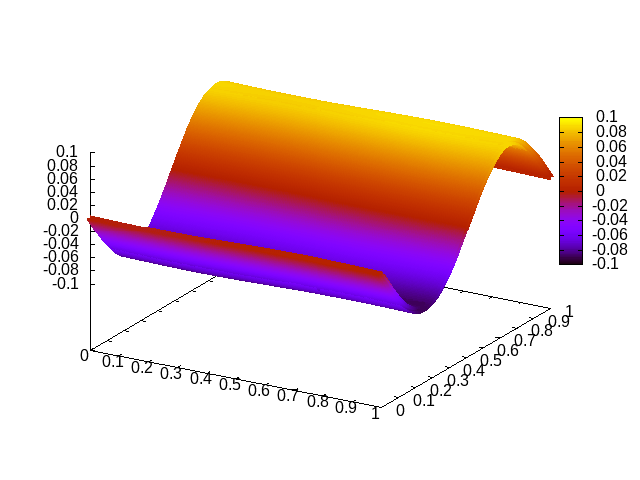}
\caption{Approximated invariant torus when $k_1 = 10^{-2}$ and $\epsilon = 10^{-3}$, after $8$ iteration of the algorithm. From left to right, up to down: 
$(\theta_1, \theta_2)\rightarrow K_j^8(\theta_1, \theta_2)$ for $j=1,\ldots,4$ (i.e., $y_1,\ y_2,\ x_1$ and $ x_2$).
 }
\label{toro}
\end{figure}

\begin{figure}[!]
\centering
\includegraphics[scale=0.35]{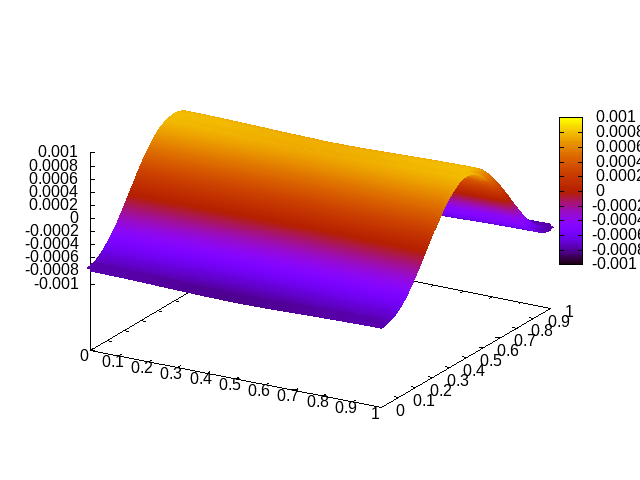}
\includegraphics[scale=0.35]{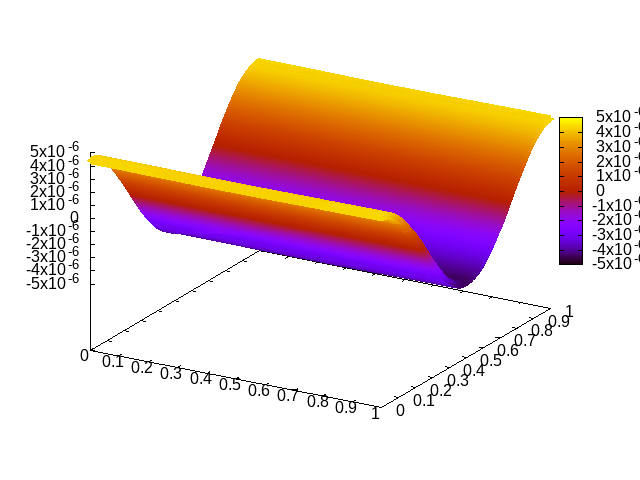}
\includegraphics[scale=0.35]{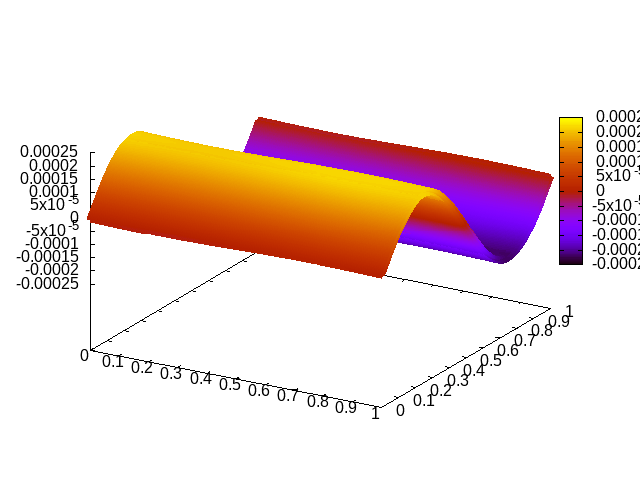}
\includegraphics[scale=0.35]{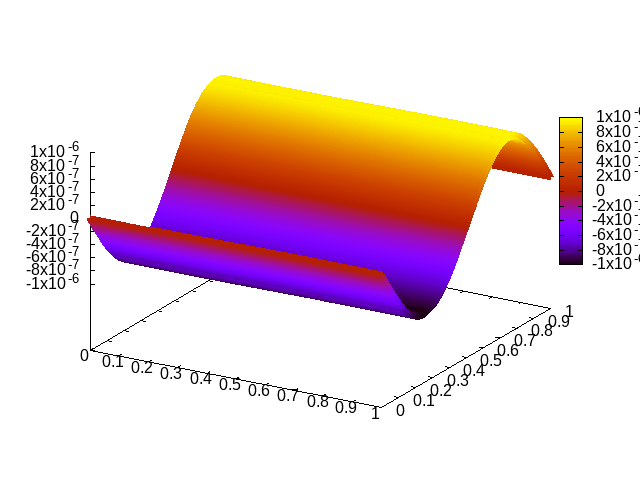}
\caption{Dynamics transverse to the torus when $k_1 = 10^{-2}$ and $\epsilon = 10^{-3}$, after $8$ iteration of the algorithm.  From left to right, up to down: 
$(\theta_1, \theta_2)\rightarrow K_j^8(\theta_1, \theta_2)$ for $j=5,\ldots,8$ (i.e., $y_3,\ y_4,\ x_3$ and $ x_4$). }
\label{lindin}
\end{figure}

\section{Conclusion and perspectives}
\label{sec:concl}

In this paper we have presented 
a method to efficiently compute lower dimensional invariant tori and their attached 
normal bundles; we tested it in a simple model, with the smallest (nontrivial) 
dimensions of the frequency vectors, that is $2$. 
The next natural step is to test its performance with more famous and challenging problems, 
such as the planetary three-body problem. Indeed, despite a numerical evidence of the existence of invariant tori, the existence of Lagrangian tori or lower dimensional tori is still an open problem, at least for real values of the masses (in the Solar System, the natural perturbing parameter is the ratio between the mass of the Sun and the one of Jupiter, i.e., $\epsilon \simeq 10^{-3}$). The latest results with real ephemerides are numerical, see \cite{FiguerasHaro3BPNum}, and for small values of the masses, eccentricities and mutual inclinations are in~\cite{Bia-chi-Val}.
For what concerns the real values of the parameters, at the present moment there are computer-assisted proofs for the secular model only, that is a model with reduced degrees of freedom, where the fast motion, i.e., the motion of revolution of the planets, is averaged. These results are based on normal forms (see~\cite{SJS} for a model of the Solar System,~\cite{upsand} for a model of an exoplanetary system with significant eccentricities and mutual inclination).

As an intermediate step before a computer-assisted proof, the performances of the algorithm can be tested also in comparison with purely numerical procedures. In fact, a fast decrease of the error in the implemented algorithm is a necessary requirement for a computer-assisted proof of the existence of invarian tori. In this context, the conversion of the semi-analytical algorithm to a computer-assisted proof is more straightforward in the parametrization method than in the normal form technique. In the latter, the algorithm and the formal proof of the convergence have to be extended with a scheme of estimates, that have to be iterated in queue of the algorithm. This is due also to the fact that the number of explicit steps which can be performed (of the order of $10^2$) is usually not yet enough to decrease the perturbation to the extreme low level of formal theorems. Even if in principle the same technique used for the KAM theory could be extended to the case of lower dimensional tori, the translation would not be straightforward. 
The faster convergence rate of the parametrization method, combined with the simplest interplay between the estimates needed to produce a formal proof and the ones needed for a computer-assisted proof, make this method more promising. However, for the mentioned reasons, the main step towards the direction of a computer-assisted proof is to produce a formal proof of the convergence of the procedure, that we will present in a forthcoming paper~\cite{CaraccioloFiguerasHaroTheory}.

\subsection*{Acknowledgments}

C.C. acknowledges the partial support from the scholarship ``Esseen, L. and C-G., for mathematical studies''.
J.-Ll.F. has been partially supported by the Swedish VR Grant 2019-04591, and 
A.H. has  been supported by the Spanish grant PID2021-125535NB-I00 (MCIU/AEI/FEDER, UE), 
and by the Spanish State Research Agency, through the
Severo Ochoa and Mar\'ia de Maeztu Program for Centers and Units of
Excellence in R\&D (CEX2020-001084-M).
Some computations were enabled by resources in project 
NAISS 2023/5-192 provided by the National Academic Infrastructure for Supercomputing in Sweden (NAISS) 
at UPPMAX, funded by the Swedish Research Council through grant agreement no. 2022-06725.

\bibliographystyle{plain}
\bibliography{biblio}
\appendix
\section{Auxiliary results}
\label{app}

As it is customary with the resolution of cohomological equations, we have the following lemma.
We omit its proof since it is standard.
\begin{lemma}
\label{lemma-russmann}
Let $v$ be $\Oscr(E)$. Then both solutions to the equations
$\Lie{\omega}{u_1}=v$ and $\Lie{\omega}{u_2}\pm\GammaOb u_2=v$ are $\Oscr(E)$. 
\end{lemma}

Now we proof a series of lemmas that help us proving Proposition \ref{lemma:Psym}.

\begin{lemma}
\label{lemma-OmegaAB}
Under the hypothesis that both $K$ and $W$ are almost invariant we have 
that $\OmegaLL, \OmegaLW $ are $\Oscr(E_K, E_W)$.
\end{lemma}
\begin{proof}
As discussed in Lemma 4.3 in~\cite{Haro-Luque-2019}, $\OmegaLL$ is $\Oscr(\EK)$ in view of the almost invariance of $K$. For what concerns the other term, let us preliminarily compute the following quantity: 
$$
\Lie{\omega}{\OmegaLW }- \OmegaLW  \Gammaab = \Lie{\omega}{L^\top}\Omega \comp K\W + L^\top\Lie{\omega}{\Omega\comp K}W + L^\top \Omega\comp K \Lie{\omega}{W}- \OmegaLW  \Gammaab.
$$
Now using the properties (which can be obtained from the invariance equation)
\begin{equation*}
\Lie{\omega}{\DK} = \Dif \EK-\DXh\comp(\K;\lambda) \DK
\end{equation*}
and
\begin{equation}\label{lieOmega}
\Lie{\omega}{\Omega\comp\K} = \DO\comp \K [\EK-\Xh\comp(\K;\lambda)]\ ,
\end{equation}
and in view of the equality~\eqref{cancellation}, one gets
\begin{align*}
\Lie{\omega}{\OmegaLW }- \OmegaLW  \Gammaab = & \,(\Dif \EK)^\top\Omega\comp K \W - L^\top (\DXh\comp(\K;\lambda))^\top \Omega\comp K \W \\& +L^\top \DO \comp(\K;\lambda) [\EK] \W - L^\top \DO\comp(\K) [\Xh\comp(\K;\lambda)] \W \\ &- L^\top\Omega\comp K \DXh\comp(\K;\lambda) W +L^\top\Omega\comp K \W \Gammaab + L^\top \Omega\comp K \EW- \OmegaLW \Gammaab \\ = &  \, (\Dif \EK)^\top\Omega\comp K \W+ L^\top \DO\comp K [\EK] \W+  L^\top \Omega\comp K \EW\,.
\end{align*}
Therefore, $\Lie{\omega}{\OmegaLW }- \OmegaLW  \Gammaab$ is $\Oscr(E_K,E_W)$ and the result for $\Omega_{LW}$ follows by Lemma \ref{lemma-russmann}.
\end{proof}

\begin{lemma}
The dummy parameter $\alpha$ is $\Oscr(E_K, E_W)$. In particular, it is equal to zero if the torus $K$ and 
the normal bundles $W$ are invariant.
\end{lemma}

\begin{proof}
One may argue that at each step we are solving an equation 
$$
\Lie{\omega}{\W}+ \DXh\comp(\K; \lambda)\W - \W \Gammaab = \EW,
$$
for a given $\alpha$, while the goal is to solve the equation for $\alpha = 0$ and thus guarantee ellipticity. 
Here we underline that at each step $\alpha$ is $\Oscr(E_K, E_W)$ and at the end of the procedure this parameter must be in fact zero. 

\noindent 
Indeed, if $W$, $K$ and $\lambda$ satisfy
$$
\Lie{\omega}{\W}+ \DXh\comp(\K; \lambda)\W - \W \Gammaab = 0,
$$
then, using the equality 
\begin{equation}
\label{cancellation}
(\DXh\comp(\K;\lambda))^\top \Omega\comp K +\DO \comp \K[\Xh\comp(\K;\lambda)]+ \Omega\comp \K \DXh\comp(\K;\lambda) = O_{2n}
\end{equation}
(see Lemma 4.3 in~\cite{Haro-Luque-2019}) we obtain that $\OmegaWW$ satisfies the following equation
\begin{align*}
\Lie{\omega}{\OmegaWW} =& \Lie{\omega} W^\top  \Omega\comp K W + W^\top\Lie{\omega}\Omega\comp \K W + W^{\top}\Omega \comp K \Lie{\omega} W \nonumber\\
=& \Gammaab^\top \Omega_{WW}+\Omega_{WW}\Gammaab+\nonumber\\
& \EW^\top\Omega\comp K W+W^\top \DO\comp K[\EW]W
+W\Omega\comp K \EW\label{eq:third eq omegaww}
\end{align*}
This can be rewritten as
\begin{equation}
\label{eq:cohoOmegaWW}
\Lie{\omega}{\OmegaWW}-\Gammaab^\top \Omega_{WW}-\Omega_{WW}\Gammaab=\hat E,
\end{equation}
where $\hat E$ is $\Oscr(E_K, E_W)$. Similarly as in solving Equation \eqref{eq:op2}
we obtain that $\Omega_{WW}-\langle\Omega_{WW}\rangle=\Oscr(\hat E)$. Moreover, 
since $\OmegaWW$ is antisymmetric it implies 
\begin{equation}
\label{eq:averOmegaWW}
\langle\Omega_{WW}\rangle=
\begin{pmatrix}
0 & -\diag(b)\\
\diag(b) & 0
\end{pmatrix}
+\Oscr(\hat E).
\end{equation}
with $b\in\mathbb R^{n-d}$.
Finally, we obtain that the averages of both sides in Equation \eqref{eq:cohoOmegaWW} are equal.
Using \eqref{eq:averOmegaWW} and writing the averages componentwise we obtain 
$2\alpha_i b_i=\Oscr(\hat E)$. But, since $b_i\neq 0$ ($\OmegaWW$ is non-degenerate), 
we get that $\alpha_i=\Oscr(\hat E)$.
\end{proof}

\begin{lemma}
Under the hypothesis that both $K$ and $W$ are almost invariant we have 
that the matrix $P(\theta)$ defined as in~\eqref{eq:P},~\eqref{L}, \eqref{N} and~\eqref{N-elem} is almost symplectic, i.e.,
\begin{equation*}
\label{almost-sym}
P^\top \OK P = \begin{pmatrix}
J_d & O_{2d\times2( n-d)} \\ O_{2(n-d)\times 2d} & \OmegaWW 
\end{pmatrix} + E_{sym}
\end{equation*}
being $E_{sym} = \Oscr(E_K,E_W)$ and, in particular, with the following expression
\begin{equation*}
\label{errsym}
E_{sym} = \begin{pmatrix}
\OmegaLL & \OmegaLL A+ \OmegaLW C & \OmegaLW \\
 -A^\top\OmegaLL -C^\top\OmegaLW ^\top & A^\top\OmegaLL A+B^\top \OmegaLL B +A^\top\OmegaLW C +C^\top \OmegaLW A & A^\top \OmegaLW \\
 -\OmegaLW ^\top & -\OmegaLW ^\top A & 0
\end{pmatrix}\,.
\end{equation*}
\end{lemma}

\begin{proof}
In view of the definition of $P$, we get the following:
$$
P^\top\Omega\comp K P = \begin{pmatrix}
\OmegaLL & \OmegaLN & \OmegaLW \\
-\OmegaLN^\top &\OmegaNN   & \OmegaNW  \\
-\OmegaLW ^\top & -\OmegaNW ^\top  & \OmegaWW 
\end{pmatrix} .
$$
Therefore, we have to check that the blocks of this matrix are of the form 
$$
\begin{pmatrix} 
O_d & -I_d & O_{d\times 2(n-d)} \\ 
I_d & O_d& O_{d\times 2(n-d)} \\
O_{2(n-d)\times d} & O_{2(n-d)\times d} &   \OmegaWW 
\end{pmatrix} + E_{sym}.
$$
We already showed in Lemma~\ref{lemma-OmegaAB} that $\OmegaLL$ is $\Oscr(\EK)$.
The second term of the first row, 
$\OmegaLN= -I_d + \OmegaLL A+ \OmegaLW C,$ is equal to $-I_d+\Oscr(E_K,E_W)$ since we 
showed in  Lemma~\ref{lemma-OmegaAB} that $\OmegaLW $ and $ \OmegaLL$ are $\Oscr(E_{K,W})$. 
The next term to be considered is 
\begin{align*}
 \OmegaNN  = &(C^\top W^\top+ B^\top L^\top  J ^\top+ A^\top L^\top)\Omega\comp K (LA+ J\comp K  L B + \W C) \\
 =& -C^\top\OmegaLW ^\top A + C^\top W^\top\Omega\comp K J\comp K  L B +C^\top\OmegaWW C
 \\ &+B^\top L^\top (J\comp K)^\top\Omega\comp K L A+ B^\top L^\top (J\comp K) ^\top\Omega\comp K J\comp K  LB \\ 
 &+ B^\top L^\top (J\comp K) ^\top \Omega\comp K \W C
  +A^\top\OmegaLL A + A^\top L^\top \Omega\comp K J\comp K  L  B + A^\top \OmegaLW C\\
 = & (A^\top\OmegaLW C)+(A^\top\OmegaLW C)^\top+ (B^\top L^\top (J\comp K) ^\top\Omega\comp K WC)-(B^\top L^\top  (J\comp K) ^\top \Omega\comp K  WC)^\top \\
 &+C^\top \OmegaWW C + (A^\top  L^\top \Omega\comp K J\comp K L  B)- (A^\top  L^\top  \Omega\comp K J\comp K L  B)^\top  + B^\top \OmegaLL B+ A^\top  \OmegaLL\,.
\end{align*}
With our choice of $A, B$ and $C$, we can rewrite 
$$
B^\top L^\top   (J\comp K) ^\top  \Omega\comp K \W C = B^\top  G_{LW}\Omega^{-1}_{WW} G_{WL}B
$$
and, therefore, we can see that it is antisymmetric in view of the antisymmetry of $\OmegaWW$ and of the resulting property $(\Omega^{-1}_{WW})^\top  = (\Omega^{T}_{WW})^{-1} = -\Omega^{-1}_{WW}$. Moreover, it is equal to $-C^\top \OmegaWW C$. Furthermore, by observing that $A= -A^\top $, we get 
\begin{align*}
\OmegaNN = &(A^\top \OmegaLW C)+(A^\top \OmegaLW C)^\top  -C^\top \OmegaWW C+2A + B^\top \OmegaLL B+ A^\top  \OmegaLL
\end{align*} 
and the result follows by the definition of $A$ and the estimates in Lemma~\ref{lemma-OmegaAB}.

\noindent
Finally, using that $G^\top  = G$ and $(\Omega^{-1}_{WW})^\top   = -\Omega^{-1}_{WW}$ and again Lemma~\ref{lemma-OmegaAB},
\begin{align*}
\OmegaNW = & (C^\top  W^\top + B^\top  L^\top   (J\comp K) ^\top + A^\top  L^\top )\Omega\comp K \W \\=&  C^\top \OmegaWW +B^\top L^\top  (J\comp K) ^\top \Omega\comp K \W + A^\top \OmegaLW \\ =&  B^\top G_{LW}(\Omega^{-1}_{WW})^\top  \OmegaWW  + B^\top G_{LW} + A^\top \OmegaLW ,
\end{align*}
and the result follows.
\end{proof}

The simplecticity of $P$ can be used to determine the inverse of $P$. However, in the case of almost simplecticity, this matrix is an approximation of the inverse of $P$, as it is seen in the following corollary.
\begin{corollary}
\label{corollary:Pinverse}
Under the hypothesis that both $K$ and $W$ are almost invariant with both $E_K$ and $E_W$ being
small enough then $P$ is invertible and has expression
\begin{equation}
\label{eq:almostinvP}
P^{-1} = \Pinv
+E_{inv}
= \begin{pmatrix}
N^\top  \Omega\comp K \\ -L^\top  \Omega	\comp K \\ \OmegaWW ^{-1} W ^\top  \Omega\comp  K 
\end{pmatrix}+E_{inv}\, ,
\end{equation}
with $E_{inv}$ being $\Oscr(E_K, E_W)$.

\end{corollary}
In the following lemma we ensure that the problem is almost reducible if both $K$ and $W$ are almost invariant.

\begin{lemma} 
\label{lemma-red}
Under the hypothesis that both $K$ and $W$ are almost invariant with both $E_K$ and $E_W$ being
small enough, the reducibility matrix $P(\theta)$ defined as 
in~\eqref{eq:P},~\eqref{L},~\eqref{N} and~\eqref{N-elem}, the problem is almost reducible, that is 
\begin{equation}
\label{eq:reducibility with error}
P^{-1}\Xscr_P= P^{-1}(\lie{\omega}{P}+ \Dz \Xh \comp (K;\lambda) P)  = \Lambda +\Oscr(E_K,E_W).
\end{equation}
\end{lemma}
\begin{proof}
Under the Lemma hypotheses we have that $P$ is almost symplectic
and satisfies Equation \eqref{eq:almostinvP}. So, the goal is to see that  
\begin{equation}
\label{eq:almostred}
P^{-1} \Xscr_P = 
\begin{pmatrix}
N^\top  \Omega\comp K \Xscr_L  & N^\top  \Omega\comp K \Xscr_N  & N^\top  \Omega\comp K \Xscr_W\\
-L^\top  \Omega\comp K \Xscr_L  & -L^\top  \Omega\comp K \Xscr_N  & -L^\top  \Omega \comp K\Xscr_W \\
\OmegaWW ^{-1} W^\top  \Omega\comp K \Xscr_L  & \OmegaWW ^{-1}W^\top  \Omega\comp K \Xscr_N  & \OmegaWW ^{-1} W^\top  \Omega \comp K \Xscr_W
\end{pmatrix} +E_{inv}\Xscr_P 
\end{equation}
is equal to 
$$
\begin{pmatrix}
O_{d} & T & O_{d\times 2(n-d)} \\ O_{d} & O_{d} & O_{d\times 2(n-d)} \\ O_{2(n-d)\times d} & O_{2(n-d)\times d}& \GammaOb
\end{pmatrix}  +\Oscr(E_K,E_W).
$$
First, because of Corollary \ref{corollary:Pinverse} we have that $E_{inv}\Xscr_P$ is $\Oscr(E_{\K}, E_W)$.
Notice that all terms in \eqref{eq:almostred} containing $\Xscr_L$ 
are $\Oscr(E_K)$ since  
$\Xscr_L= \lie{\omega}{L} + \DXh \comp (K;\lambda)L = \Dif  \EK$ and 
$\Dif \EK$ is $\Oscr(E_K)$ because of Cauchy inequality for analytic functions.

For the second column, the first term is equal to $T$ by definition.
The second term is 
$$
-L^\top \Omega\comp K \Xscr_N =  -L^\top  \Omega\comp K  \Lie{\omega}N -L^\top \Omega\comp K  \DXh\comp (K;\lambda) N.
$$
Using that 
$\lie{\omega}{\OmegaLN}= \lie{\omega}{L^\top }\Omega\comp K  N + L^\top \lie{\omega}{\Omega\comp K }N + L^\top \Omega\comp K  \lie{\omega}{N}$ and the formul{\ae}~\eqref{lieOmega} and~\eqref{cancellation}, one obtains
\begin{align*}
-L^\top \Omega \comp K \Xscr_N= & \lie{\omega}{L^\top }\Omega\comp K  N + L^\top \lie{\omega}{\Omega\comp K }N -\lie{\omega}{\OmegaLN}-L^\top \Omega\comp K  \DXh \comp (K;\lambda) N\\
=& -L^\top  (\DXh\comp (\K;\lambda))^\top \Omega\comp K  N + \Dif \EK^\top \Omega\comp K  N+ L^\top  \DO\comp K [\EK]N\\&- L^\top  \DO\comp K  \Xh\comp (K;\lambda)  N  
-\lie{\omega}{\OmegaLN}-L^\top \Omega\comp K  \DXh \comp (K;\lambda) N\\ =& \Dif \EK^\top  \Omega \comp K  N -\lie{\omega}{\OmegaLN}+ L^\top  \DO\comp K [\EK] N
\end{align*}
and it is, therefore, of $\Oscr(E_K,E_W)$.
Similarly, we can show that 
\begin{align*}
\OmegaWW ^{-1} W^\top  \Omega \comp K \Xscr_N = & \OmegaWW ^{-1}(-\lie{\omega}{W^\top }\Omega \comp K N - W^\top \lie{\omega}\Omega\comp K  N +\lie{\omega}{\OmegaWN} + W^\top \Omega \comp K \DXh \comp (K;\lambda) N)\\ =& -\OmegaWW ^{-1}(\EW^\top \Omega\comp K  N -\Gammaab^\top  \OmegaWN- W^\top  \DO\comp K [\EK] N + \lie{\omega}{\OmegaWN})
\end{align*}
and, therefore, it is $\Oscr(E_{K,W})$.

The third column contains the following terms:
$$
N^\top  \Omega\comp K  \Xscr_W = N^\top  \Omega(\lie{\omega}{W}+ \Dz \Xh\comp (K;\lambda)  W) = N^\top  \Omega \comp K(W\Gammaab + \EW) = \OmegaNW \Gammaab + N^\top \Omega \comp K \EW\,,
$$
$$
-L^\top  \Omega \comp K \Xscr_W = -\OmegaLW \Gammaab - L^\top \Omega\comp K  \EW\,,
$$
$$
\OmegaWW ^{-1} W^\top  \Omega \comp K \Xscr_W = \Gammaab + \OmegaWW ^{-1} W^\top  \Omega \comp K \EW\,=\GammaOb + \Gamma_{\alpha,0} + \OmegaWW ^{-1} W^\top  \Omega  \comp K\EW .
$$
Therefore, using the fact that $\OmegaLW $, $\OmegaNW$ and $\alpha$ are $\Oscr(E_K,E_W)$ gives
us the desired smallness.
\end{proof}

\end{document}